\documentclass[12pt]{article}

\usepackage{amssymb,amsmath,amsthm}

\usepackage{hyperref}

\DeclareMathOperator{\td}{td}  %transcendence degree
\newcommand{\Z}{\ensuremath{\mathbb{Z}}}
\newcommand{\Q}{\ensuremath{\mathbb{Q}}}

\newcommand{\C}{\ensuremath{\mathbb{C}}}
\renewcommand{\L}{\ensuremath{\mathcal{L}}} % language L

\renewcommand{\phi}{\varphi}

\newcommand{\into}{\hookrightarrow}

\newcommand{\nstrong}{\ensuremath{\not\kern-4pt\lhd\;}} % nonstrong embedding

\newbox\noforkbox \newdimen\forklinewidth
\forklinewidth=0.3pt \setbox0\hbox{$\textstyle\smile$}
\setbox1\hbox to \wd0{\hfil\vrule width \forklinewidth depth-2pt
  height 10pt \hfil}
\wd1=0 cm \setbox\noforkbox\hbox{\lower 2pt\box1\lower
2pt\box0\relax}
\def\unionstick{\mathop{\copy\noforkbox}\limits}

\def\nonfork_#1{\unionstick_{\textstyle #1}}

\setbox0\hbox{$\textstyle\smile$} \setbox1\hbox to \wd0{\hfil{\sl
/\/}\hfil} \setbox2\hbox to \wd0{\hfil\vrule height 10pt depth
-2pt width
                \forklinewidth\hfil}
\newbox\doesforkbox
\setbox\doesforkbox\hbox{\lower 2pt\box1 \lower
2pt\box2\lower2pt\box0\relax}
\def\nunionstick{\mathop{\copy\doesforkbox}\limits}

\def\fork_#1{\nunionstick_{\textstyle #1}}

\DeclareMathOperator{\height}{height}
\DeclareMathOperator{\depth}{depth}

\newtheorem{prop}{Proposition}
\newtheorem{cor}[prop]{Corollary}
\newtheorem{theorem}[prop]{Theorem}
\newtheorem{lemma}[prop]{Lemma}

\newtheorem{conjecture}[prop]{Conjecture}

\theoremstyle{definition}
\newtheorem{defn}[prop]{Definition}
\newtheorem{example}[prop]{Example}

\newcommand{\y}{\bar{y}}

\newcommand{\then}{\Rightarrow}

\title{Schanuel's Conjecture and Algebraic Roots of Exponential Polynomials}

\author{Ahuva C. Shkop}

\begin{document}

\maketitle

\begin{abstract}
In this paper, I will prove that assuming Schanuel's conjecture, an exponential polynomial with algebraic coefficients
can have only finitely many algebraic roots.  Furthermore, this proof demonstrates that there are no unexpected algebraic
roots of any exponential polynomial.  This implies a special case of Shapiro's conjecture: if $p(x)$ and $q(x)$ are two such  exponential polynomials with algebraic
coefficients which have common factors only of the form $\exp(g)$ for some exponential polynomial $g$, $p$ and $q$
have only finitely many common zeros.
\end{abstract}

\section{Introduction}

In the 1960's, Schanuel made the following conjecture:

\begin{conjecture} If $\{z_1,...,z_n \}\subset \C$, then $\td_\Q(z_1,...,z_n,e^{z_1},...,e^{z_n})$, where $\td_\Q$ is the transcendence 
degree over $\Q$, is at least the
 $\Q$ linear dimension of $\{z_1,...,z_n\}$

\end{conjecture}

While there are proofs of special cases of this statement (e.g. Lindemann-Weierstrass Theorem),
Schanuel's conjecture is as yet unproven.  In \cite{P-exp}, Zilber constructs an algebraically closed exponential
field known as pseudoexponentiation which satisfies the analog of Schanuel's conjecture.  We will make use of the following
generalization of Schanuel's conjecture.

\begin{defn}
An algebraically closed exponential field $K$ satisfies Schanuel's conjecture if for any $\{z_1,...,z_n \}\subset K$, the
$\td_\Q(z_1,...,z_n,\exp(z_1),...,\exp(z_n))$ is at least the
 $\Q$ linear dimension of $\{z_1,...,z_n\}$
\end{defn}

In this paper we will
give various consequences of Schanuel's conjecture.  Since Zilber's construction satisfies this as well as the more general conditions we will set on the 
algebraically closed exponential field, these results are theorems of pseudoexponentiation.

We will use the notation $\Q^{alg}$ to refer to the algebraic closure of the rational numbers. The goal of this paper is to prove the following theorem:

\begin{theorem}\label{Thm 1} Suppose $p(x)$ is an exponential polynomial in $\Q^{alg}[x]^E$.
 Then Schanuel's conjecture implies that $p(x)$ has finitely many algebraic zeros.
\end{theorem}

We will define $\Q^{alg} [x]^E$ in the following section.

\section{The Exponential Polynomial Ring}

We will begin with the following definitions.

\begin{defn}
In this paper, a \emph{(total) E-ring} is a $\Q$-algebra $R$ with
  no zero divisors, together with a homomorphism
  $\exp: \langle R,+ \rangle \to \langle R^*,\cdot \rangle$.

  A \emph{partial E-ring} is a $\Q$-algebra $R$ with no zero
  divisors, together with a $\Q$-linear subspace $A(R)$ of $R$ and a
  homomorphism $\exp: \langle A(R),+ \rangle \to \langle R^*,\cdot \rangle$. $A(R)$ is then the domain of $\exp$.

  An  \emph{E-field} is an E-ring which is a field.

  We say \emph{$S$ is a partial E-ring extension of $R$} if $R$ and $S$ are partial E-ring, $R\subset S$, and for all $r\in A(R)$,
  $\exp_S(r)=\exp_R(r)$.
\end{defn}
\noindent Recall the following construction of $K[X]^E$, the exponential polynomial ring over an E-field $K$ on the set of
indeterminates $X$: (see \cite{Lou E-rings},\cite{Angus E-rings})

\bigskip

If $R$ is a partial E-ring, we can construct $R'$, a partial E-ring extension of $R$, with the following properties:
\begin{itemize}
\item The domain of the exponential map in $R'$ is precisely $R$.
\item The kernel of the exponential map in $R'$ is precisely the kernel of the exponential map in $R$.
\item If $y_i \notin A(R)$ for $i=1,...,n$, then $\td_R(\exp_{R'}(\y))$ in $R'$ will be exactly the $\Q$-linear dimension of $\y$ over
    $A(R)$.
\item $R'$ is generated as a ring by $R\cup \exp(R)$.
\end{itemize}

For $K$ an E-field and $X$ a set of indeterminates, let $K[X]$ be the partial E-ring where $A(K[X])= K$. Then the exponential polynomial ring over $K$, $K[X]^E$,  is simply the union of the chain

 \[ K[X]= R_0 \into R_1 \into R_2 \into R_3 \into R_4  \into \cdots \]

where $R_{n+1} = R_n'$.

\medskip

\noindent This construction yields a natural notion of height.

\begin{defn}
For $p$ an exponential polynomial  $$\height(p)= \min \{ i: p\in R_i\}$$
\end{defn}

\begin{example}
The exponential polynomial $p(x_1,x_2)= \exp(\exp(\frac{x_1}{2} + x_2^2)) + x_1^3$ in $\C[x_1,x_2]^E$  has
height $2$.

\end{example}

Fix $K$, an algebraically closed exponential field. Let $\Q^{alg}[x]_E$ be the exponential subring of $K[x]^E$
generated by $\Q^{alg}[x]$.

\bigskip

\noindent \begin{bf}NOTE\end{bf}: The definition of $\Q^{alg}[x]_E$ depends entirely on $K$. We are fixing this algebraically closed exponential field at
this point to avoid cumbersome notation.  When we assume Schanuel's conjecture, we are assuming $K$ satisfies Schanuel's conjecture
as in the introdution.  Since we have not as of yet specified anything about the exponential map on $K$, it
is worth noting that $\Q^{alg}$ may not be an exponential field (as in the case $K = \C_{\exp}$) in which case $\Q^{alg}[x]_E$
is \emph{not} an exponential polynomial ring over $\Q^{alg}$.  Thus, when we refer to the height of an element of
$\Q^{alg}[x]_E$, we refer to its height in $K[x]^E$.

\bigskip

The following lemma is easy and useful.
\begin{lemma} Let $\widehat{\Q}$ be the union of the chain $$Q_0 \into Q_1 \into \cdots$$ where
$Q_0 = \Q^{alg}$ and $Q_{i+1} = [Q_i \cup \exp(Q_i)]$, the subring of $K$ generated by $Q_i$
and the exponential image of $Q_i$.

Clearly $\widehat{\Q}$ is an E-ring.

Then $\Q^{alg}[x]_E = \widehat{\Q}[x]^E$, the free E-ring over $\widehat{\Q}$ on $x$.

\end{lemma}

This filtration of $\widehat{\Q}$ yields a natural notion of depth.

\begin{defn}
Let $q \in \widehat{\Q}$.  Then $$\depth(q) = \min\{i: q \in Q_i\}.$$
\end{defn}

\begin{lemma}
Let $s \in Q_i$.

Then there are $q_1,\dots,q_n \in \Q^{alg}$ and $s_1,\dots,s_m \in Q_{i-1}$ such that

\begin{itemize}
\item For all $1\leq i \leq m$, $s_i$ is algebraic over the set $$ \{q_1,...,q_n,\exp(q_1),...,\exp(q_n),\exp(s_1),\exp(s_n)\}$$
\item $s$ is algebraic over $$\{q_1,...,q_n,s_1,...,s_m,\exp(q_1),...,\exp(q_n),\exp(s_1),\exp(s_m)\}$$
\item $q_1,...,q_n,s_1,...,s_m$ are $\Q$-linearly independent.
\end{itemize}

\end{lemma}

\begin{proof}
This proof is an easy induction on $i$.

\end{proof}

\begin{lemma}
Assume Schanuel's conjecture. Then the exponential map on $\Q^{alg}[x]_E$ is injective.
\end{lemma}

\begin{proof}
Since free constructions do not add to the kernel, it suffices to show that the kernel of the exponential map on $\widehat{\Q}$ is
$\{0\}$. To accomplish this, we will induct on depth.

\noindent Suppose that $q \in \Q^{alg}$ and $\exp(q)=1$.  Then $td_\Q(q,\exp(q))= 0$ and by Schanuel's conjecture, $q=0$.  Thus
$Q_0 \cap \ker(\exp_{\widehat{\Q}}) = \{0 \}$.
Suppose for purposes of induction that  $Q_i \cap \ker(\exp_{\widehat{\Q}}) =\{ 0\}$.

\noindent Now suppose that $s\in Q_{i+1} \cap \ker(\exp_{\widehat{\Q}})$.

Let $q_1,...,q_n, s_1,...,s_m$ be as above.  Then since for all $j$, $s_j$ is algebraic over
$\{q_1,...,q_n, \exp(q_1),\dots,\exp(q_n), \exp(s_1),\dots,\exp(s_m)\}$, $s$ is algebraic over $\{q_1,...,q_n,s_1,...,s_m,\exp(q_1),...,\exp(q_n),\exp(s_1),\exp(s_m)\}$,
and $\exp(s) = 1$, we have  $$\td(q_1,...,q_n,s_1,...,s_m,s,\exp(q_1),...,\exp(q_n),\exp(s_1),\dots,\exp(s_m),\exp(s))$$
$$= \td(q_1,...,q_n,\exp(q_1),...,\exp(q_n),\exp(s_1),...,\exp(s_m)).$$

\noindent Thus, since $q_1,...,q_n \in \Q^{alg}$, we have
  $$\td(q_1,...,q_n,s_1,...,s_m,s,\exp(q_1),...,\exp(q_n),\exp(s_1),\dots,\exp(s_m),\exp(s))$$$$ \leq n+m $$

Thus, since $q_1,...,q_n,s_1,...,s_m$ are $\Q$-linearly independent, Schanuel's conjecture implies that $s$ is $\Q$-linearly dependent on $q_1,...,q_n,s_1,...,s_m$ which in turn
implies that $s \in Q_{n}$.  But we assumed that $Q_n \cap \ker(\exp_{\widehat{\Q}})= \{0\}$.  Thus $s=0$.

\end{proof}

\section{Decompositions of p}
We now fix an exponential polynomial $p(x) \in \Q^{alg}[x]_E$.

\begin{defn} We will call a set $T$ of exponential polynomials a \emph{decomposition of $p$} if it is a minimal set of
exponential polynomials such that:

\begin{itemize}
\item $\exists t_1,...,t_k \in T : p \in \Q^{alg}[x,\exp(t_1),...,\exp(t_k)]$, the subring of $\Q^{alg}[x]_E$ generated by
$x, \exp(t_1),...,\exp(t_k).$
\item $t_i \in T \then \exists t_1,...,t_l \in T : t_i \in \Q^{alg}[x,\exp(t_1),...,\exp(t_l)].$
\item There is an $L \in \Z^*$ such that $\frac{x}{L} \in T$.
\end{itemize}

\end{defn}
\noindent We will call elements of $T$ \emph{$T$-bricks}.

 Let $p^* \in \Q^{alg}[x,\y]$ be such that $p^*(x,\exp(x),\exp(t_{1}),...,\exp(t_\alpha)))=p$ as in the first part of the definition.
\medskip

Consider the parallel between exponential polynomials and terms in the language
$\L = \{+,-,\cdot, 0,1, \exp\} \cup \{c_k : k \in \Q^{alg}\}$.  This parallel extends to subterms and $T$-bricks.
 Considering this parallel, notice that every $T$-brick can be written as a polynomial in
 $x$ and the exponential image of the $T$-bricks of lower height.  Furthermore, all decompositions are finite. To
 satisfy the third bullet consider the following: While there are several terms which correspond to the same polynomial,
 we can choose one such term and take the least common multiple of the denominators of the rational
 coefficients of all the elements of $x$ which appear in the term.

\begin{example}
Consider $p(x)= \exp(\exp(\frac{x}{2}+ x^2)) +x^3$. Then $T = \{\frac{x}{2}, x^2, \exp(\frac{x}{2}+
x^2)\}$ is a decomposition of $p$. Notice that $\frac{x}{2}+ x^2$ is not in the decomposition since
$\exp(\frac{x}{2}+ x^2) = \exp(\frac{x}{2})\exp(x^2)$.
\end{example}

\medskip

\begin{defn}  We say that a decomposition $T$ is a \emph{refined decomposition} if $T$ is $\Q-$linearly independent over
$\Q^{alg}$.
\end{defn}

Recall the following fact: (See \cite{Me1})
\begin{lemma}  Given a decomposition $T$, we can form a refined decomposition $T'$.
\end{lemma}

We now fix a refined decomposition $T$ of $p$, and let $p^*$ witness this as above. Elements of $T$ will be called $t_i$ for
$1 \leq i \leq \alpha$ where $t_i \neq t_j$ for $i \neq j$ and $|T| = \alpha$.

\section{Collapsing Points}

Let $a_i(x)\in \Q^{alg}[x]$ be nonzero polynomials and $g_i(x)$ exponential polynomials in $\Q^{alg}[x]_E$ such that $g_i(x) \neq g_j(x)$ for $i\neq j$ and  $$p(x) = \sum_{i=1}^{m} a_i(x)\exp(g_i(x)).$$

Fix these choices of $a_i$ and $g_i$.

\begin{defn}  We say \emph{$p$ collapses at $\beta$} if either $a_i(\beta)=0$ for all $i$, or there is some $i,j$ , $i
\neq j$ and $g_i(\beta)= g_j(\beta)$.
\end{defn}

\begin{theorem} Suppose $\beta \in \Q^{alg}$ and $p(\beta)=0$.  Then Schanuel's conjecture implies that $p$ collapses at $\beta$.

\end{theorem}
\begin{proof}
To begin this proof, we will need to set some notation.

Let $p, g_i,$ and $a_i$ be as above.
The we have $$p(x) =  \sum_{i=1}^{m} a_i(x)\exp(g_i(x))= p^*(x,\exp(t_1(x)),...,\exp(t_\alpha(x))).$$

We also have $$p^*(x,Y_1,...,Y_\alpha) = \sum_{i=1}^{m} a_i(x)\psi_i(Y_1,...,Y_\alpha)$$ and $$p^*(\beta,Y_1,...,Y_\alpha) = \sum_{i=1}^{m} a_i(\beta)\psi_i(Y_1,...,Y_\alpha)$$
where $\psi_i$ is a monomial for all $i$.
Notice that since $g_i(x) \neq g_j(x)$ for $i\neq j$, $\psi_i(Y_1,...,Y_\alpha) \neq \psi_j(Y_1,...,Y_\alpha)$ for $i\neq j$.

\bigskip

\noindent Now suppose that for some $1\leq i \leq m$, $a_i(\beta) \neq 0$.  Then $p^*(\beta, Y_1,...,Y_\alpha) \neq 0$.
Since $p^*(\beta,\exp(t_1(\beta)),...,\exp(t_\alpha(\beta)))=0$ we know that for some $d < \alpha,$
 $$\td(\exp(t_1(\beta)),...,\exp(t_\alpha(\beta)))\leq d.$$

Since $\beta$ is algebraic and the $T$-bricks are algebraic over $x$ and the exponential image of $T$, we know that

$$\td(\exp(t_1(\beta)),...,\exp(t_\alpha(\beta))) = \td(\beta, t_1(\beta),...,t_\alpha(\beta), \exp(t_1(\beta)),...,\exp(t_\alpha(\beta))).$$

Since $\beta$ is algebraic, this is equal to

$$ \td( t_1(\beta),...,t_\alpha(b), \exp(t_1(\beta)),...,\exp(t_\alpha(\beta)))$$

and we get

$$\td( t_1(\beta),...,t_\alpha(b), \exp(t_1(\beta)),...,\exp(t_\alpha(\beta)))=d < \alpha.$$

\noindent Assuming Schanuel's conjecture, we know that the $\Q$-linear dimension of of $\{t_1(\beta),...,t_\alpha(\beta)\}$ is at most  $d.$  Thus we can reorder the $T$-bricks so that for all $\alpha\geq j>d$,we have $$t_j(\beta) = \sum_{i=1}^{d} \frac{ m_{j,i}}{L}t_i(\beta)$$
where $m_{i,j} \in \Q.$

Now let $r(x,Y_1,...,Y_d)$ be the polynomial with rational exponents such that
$$r(x,Y_1,...,Y_d) = p^*\Big(x,Y_1,...,Y_d, \prod_{i=1}^d Y_i^{m_{i,d+1}},...., \prod_{i=1}^d Y_i ^{m_{i,\alpha}}\Big).$$

So $$r(x,Y_1,...,Y_d) = \sum_{i=1}^m a_i(x)\phi(Y_1,...,Y_d)$$ where for each $i$, $$\phi_i(Y_1,...,Y_d) = \psi_i\Big(Y_1,...,Y_d,\prod_{i=1}^d Y_i^{m_{i,d+1}},...., \prod_{i=1}^d Y_i ^{m_{i,\alpha}}\Big)$$ and is thus of the form $\prod Y_i^{q_i}$ for some $q_i \in \Q$.

Now we must compile all the information we have.

$$r(\beta, \exp(t_1(\beta)),...,\exp(t_d(\beta)))$$ $$= p^*\Big(\beta, \exp(t_1(\beta)),...,\exp(t_d(\beta)),\prod_{i=1}^{d} \exp(m_{i,d+1}t_i(\beta)),...,\prod_{i=1}^{d} \exp(m_{i,\alpha}t_i(\beta))\Big)$$
$$= p^*(\beta,\exp(t_1(\beta)),...,\exp(t_\alpha(\beta))$$ $$=p(\beta) = 0.$$

Since $\{\exp(t_1(\beta)),...,\exp(t_d(\beta))\}$ is algebraically independent, we know $r(\beta, Y_1,...,Y_d)=0$.

Since $r(\beta,Y_1,..., Y_d) = \sum_{i=1}^m a_i(\beta)\phi(Y_1,...,Y_d)$ and for some $i$, $a_i(\beta)\neq 0$ we know that
$\phi_i(Y_1,...,Y_d) = \phi_j(Y_1,...,Y_d)$ for some $i \neq j$.  But then $\phi_i (\exp(t_1(\beta)),...,\exp(t_d(\beta))) = \phi_j(\exp(t_1(\beta)),...,\exp(t_d(\beta)))$ and we have

$$\psi_i\Big(\exp(t_1(\beta)),...,\exp(t_d(\beta)),\prod_{k=1}^d \exp(m_{k,d+1}t_k),...,\prod_{k=1}^d \exp(m_{k,\alpha}t_k)\Big)$$$$=
 \psi_j\Big(\exp(t_1(\beta)),...,\exp(t_d(\beta)),\prod_{k=1}^d \exp(m_{k,d+1}t_k),...,\prod_{k=1}^d \exp(m_{k,\alpha}t_k)\Big)$$

and

$$\psi_i(\exp(t_1(\beta)),...,\exp(t_\alpha(\beta))) = \psi_j(\exp(t_1(\beta)),...,\exp(t_\alpha(\beta))).$$

Thus $\exp(g_i(\beta)) = \exp(g_j(\beta))$ and since $\exp$ is injective on $\Q^{alg}[x]_E$, $g_i(\beta) = g_j(\beta)$.

\end{proof}

\begin{cor}
Assume Schanuel's conjecture. Then if $p \in \Q^{alg}[x]_E$, then $p$ has only finitely many algebraic zeros.

\end{cor}

\begin{proof}This is a simple induction on height.

\noindent Base case: $\height(p)=0$.  The $p$ is a polynomial in one variable and has only finitely many zeros.
\noindent inductive step:  Suppose $p(\beta) =0$.  Then $p$ collapses at $\beta$.  So $g_i = g_j$ for some $i\neq j$ or $a_i(\beta)=0$ for all $i$.
Each of these options implies that $\beta$ is a zero of one of finitely many nontrivial exponential polynomials of lower height.  Each of these have only finitely many
algebraic zeros by induction.

\end{proof}

\begin{cor} Assume Schanuel's conjecture. Let $p,q \in \Q^{alg}[x]_E$ so that $p$ and $q$ have no common factors aside from units and are both of height 1.
Then $p$ and $q$ have only finitely many common zeros.
\end{cor}

\begin{proof}
Let $T = \{t_1(x),...,t_\alpha(x)\}$ be a refined decomposition of both $p$ and $q$.  (Simply require that both can be constructed using $T$).
Suppose $p(\beta) = q(\beta) = 0$.  For some integer $L$, we know that  $\frac{x}{L}$ is one of the $T$-bricks and that every other
$T$-brick is algebraic over this $T$-brick and the exponential image of $T$. Therefore we know that

 $$\td(t_1(\beta),...,t_\alpha(\beta), \exp(t_1(\beta)),...,\exp(t_\alpha (\beta)))= \td(\frac{\beta}{L},\exp(t_1(\beta)),...,\exp(t_\alpha (\beta))) $$

 It is clear than a common factor of $p^*$ and
$q^*$ would imply a common factor of $p$ and $q$. Therefore, since $p$ and $q$ have no common factors, $p^*$ and $q^*$ have no common
factors. Thus, since $$p^*(\beta, \exp(t_1(\beta)),...,\exp(t_\alpha (\beta))) = q^*(\beta, \exp(t_1(\beta)),...,\exp(t_\alpha (\beta)))= 0$$
  Thus, we have that

  $$\td(\beta,\exp(t_1(\beta)),...,\exp(t_\alpha (\beta)))\leq \alpha-1 $$

  and thus
$$\td(t_1(\beta),...,t_\alpha(\beta), \exp(t_1(\beta)),...,\exp(t_\alpha (\beta))) \leq \alpha-1.$$

Thus, Schanuel's conjecture implies that $t_1(\beta),...,t_\alpha(\beta)$ are $\Q$-linearly dependent.  Since $T$ is a refined decomposition, and comprised of polynomial $T$-bricks,
we can deduce that $t_1(\beta),...,t_\alpha(\beta)$ satisfy a non-trivial $\Q$ linear polynomial, and that $\beta$ satisfies a
nontrivial polynomial over $\Q^{alg}$.  Thus, $\beta$ is algebraic. By above, there are only finitely many algebraic zeros.
\end{proof}

\begin{cor}\emph{\bf{(Shapiro's conjecture over the algebraic numbers.)}}
Assume Schanuel's conjecture.  Suppose  $$p(x)= \sum_{i=1}^n a_i\exp(b_ix)$$ and
$$q(x)= \sum_{i=1}^m c_i\exp(d_ix)$$ where the $a_i,b_i,c_i,d_i \in \Q^{alg}$.   Then, if $p(x)$ and $q(x)$ have no common factors
aside from units, $p(x)$ and $q(x)$ have only finitely many common zeros.

\end{cor}

\end{document}